\documentclass[a4paper,10pt]{amsart}

\usepackage{microtype}
\usepackage{fullpage}
\usepackage{amsmath,amsfonts,amssymb}
\usepackage[utf8x]{inputenc}
\usepackage[english]{babel}
\usepackage[T1]{fontenc}
\usepackage{lmodern}
\usepackage{graphicx}
\usepackage[all]{xy}
\usepackage{latexsym}
\usepackage{fancyhdr}
\usepackage{lscape}
\usepackage{url}
\usepackage{ifthen}
\usepackage{booktabs}
\usepackage{paralist}
\usepackage{array}
\usepackage{multirow}
\usepackage{epstopdf}
\usepackage{ifpdf}
\usepackage{caption}
\usepackage{subcaption}
\usepackage{multicol}
\usepackage{color}

\usepackage[hypertexnames=false,
backref=page,
    pdftex,
    pdfpagemode=UseNone,
    breaklinks=true,
    extension=pdf,
    colorlinks=true,
    linkcolor=blue,
    citecolor=blue,
    urlcolor=blue,
]{hyperref}

\ifpdf
\fi

 \graphicspath{./}

\newtheorem{prop}{Proposition}[section]
\newtheorem{thm}[prop]{Theorem}

\theoremstyle{definition}

\newcommand{\N}{\mathbb{N}}
\newcommand{\Z}{\mathbb{Z}}

\newcommand{\R}{\mathbb{R}}
\newcommand{\C}{\mathbb{C}}

\newcommand{\sspace}{\cdot}
\newcommand{\ssspace}{\cdot\cdot}

\DeclareMathOperator{\im}{i}
\DeclareMathOperator{\imm}{im}
\DeclareMathOperator{\de}{d}

\newcommand{\del}{\partial}
\newcommand{\delbar}{\overline{\del}}

\title{On Bott-Chern cohomology of compact complex surfaces}

\author{Daniele Angella}
\address[Daniele Angella]{Istituto Nazionale di Alta Matematica}
\curraddr{Dipartimento di Matematica e Informatica\\
Universit\`{a} di Parma\\
Parco Area delle Scienze 53/A, 43124\\
Parma, Italy}
\email{daniele.angella@math.unipr.it}

\author{Georges Dloussky}
\address[Georges Dloussky]{Aix-Marseille University, CNRS, Centrale Marseille, I2M, UMR 7373, 13453\\
39 rue F. Joliot-Curie 13411\\
Marseille Cedex 13, France}
\email{georges.dloussky@univ-amu.fr}

\author{Adriano Tomassini}
\address[Adriano Tomassini]{Dipartimento di Matematica e Informatica\\
Universit\`{a} di Parma\\
Parco Area delle Scienze 53/A, 43124\\
Parma, Italy}
\email{adriano.tomassini@unipr.it}

\keywords{compact complex surfaces, Bott-Chern cohomology, class VII, solvmanifold}
\thanks{This work was supported by the Project PRIN ``Varietà reali e complesse: geometria, topologia e analisi armonica'', by the Projects FIRB ``Geometria Differenziale Complessa e Dinamica Olomorfa'' and ``Geometria Differenziale e Teoria Geometrica delle Funzioni'', by GNSAGA of INdAM, and by ANR ``Méthodes nouvelles en géométrie non kählerienne''.\\\\
\indent To appear in {\em Annali di Matematica Pura ed Applicata}.
The final publication is available at Springer via \url{http://dx.doi.org/10.1007/s10231-014-0458-7}.
}
\subjclass[2010]{32C35, 57T15, 32J15}

\begin{document}

\begin{abstract}
We study Bott-Chern cohomology on compact complex non-K\"ahler surfaces. In particular, we compute such a cohomology for compact complex surfaces in class $\text{VII}$ and for compact complex surfaces diffeomorphic to solvmanifolds.
\end{abstract}

\maketitle

\section*{Introduction}
For a given complex manifold $X$, many cohomological invariants can be defined, and many are known for compact complex surfaces.

Among these, one can consider {\em Bott-Chern and Aeppli cohomologies}. They are defined as follows:
$$ H^{\bullet,\bullet}_{BC}(X) \;:=\; \frac{\ker\del\cap\ker\delbar}{\imm\del\delbar} \qquad \text{ and } \qquad H^{\bullet,\bullet}_{A}(X) \;:=\; \frac{\ker\del\delbar}{\imm\del+\imm\delbar} \;. $$

Note that the identity induces natural maps
$$
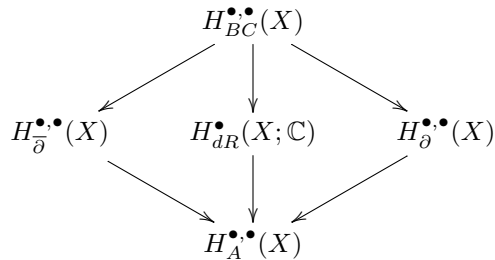
\xymatrix{
 & H^{\bullet,\bullet}_{BC}(X) \ar[ld] \ar[d] \ar[rd] & \\
H^{\bullet,\bullet}_{\delbar}(X) \ar[rd] & H^{\bullet}_{dR}(X;\C) \ar[d] & H^{\bullet,\bullet}_{\del}(X) \ar[ld] \\
& H^{\bullet,\bullet}_{A}(X) &
}
$$
where $H^{\bullet,\bullet}_{\delbar}(X)$ denotes the Dolbeault cohomology and $H^{\bullet,\bullet}_{\del}(X)$ its conjugate, and the maps are morphisms of (graded or bi-graded) vector spaces. For compact K\"ahler manifolds, the natural map $\bigoplus_{p+q=\bullet} H^{p,q}_{BC}(X) \to H^{\bullet}_{dR}(X;\C)$ is an isomorphism.

Assume that $X$ is compact. The Bott-Chern and Aeppli cohomologies are isomorphic to the kernel of suitable $4$th-order differential elliptic operators, see \cite[\S2.b, \S2.c]{schweitzer}. In particular, they are finite-dimensional vector spaces. In fact, fixed a Hermitian metric $g$, its associated $\C$-linear Hodge-$*$-operator induces the isomorphism
$$ H^{p,q}_{BC}(X) \stackrel{\simeq}{\to} H^{n-q,n-p}_{A}(X) \;, $$
for any $p,q\in\{0,\ldots,n\}$, where $n$ denotes the complex dimension of $X$. In particular, for any $p,q\in\{0,\ldots,n\}$, one has
$$ \dim_\C H^{p,q}_{BC}(X) \;=\; \dim_\C H^{q,p}_{BC}(X) \;=\; \dim_\C H^{n-p,n-q}_{A}(X) \;=\; \dim_\C H^{n-q,n-p}_{A}(X) \;. $$

For the Dolbeault cohomology, the Fr\"olicher inequality relates the Hodge numbers and the Betti numbers: for any $k\in\{0,\ldots,2n\}$,
$$ \sum_{p+q=k} \dim_\C H^{p,q}_{\delbar}(X) \;\geq\; \dim_\C H^k_{dR}(X;\C) \;. $$
Similarly, for Bott-Chern cohomology, the following inequality {\itshape à la} Fr\"olicher has been proven in \cite[Theorem A]{angella-tomassini-3}: for any $k\in\{0,\ldots,n\}$,
$$ \sum_{p+q=k} \left( \dim_\C H^{p,q}_{BC}(X) + \dim_\C H^{p,q}_{A}(X) \right) \;\geq\; 2\, \dim_\C H^k_{dR}(X;\C) \;. $$

The equality in the Fr\"olicher inequality characterizes the degeneration of the Fr\"olicher spectral sequence at the first level. This always happens for compact complex surfaces. On the other side, in \cite[Theorem B]{angella-tomassini-3}, it is proven that the equality in the inequality {\itshape à la} Fr\"olicher for the Bott-Chern cohomology characterizes the validity of the {\em $\del\delbar$-Lemma}, namely, the property that every $\del$-closed $\delbar$-closed $\de$-exact form is $\del\delbar$-exact too, \cite{deligne-griffiths-morgan-sullivan}. The validity of the $\del\delbar$-Lemma implies that the first Betti number is even, which is equivalent to K\"ahlerness for compact complex surfaces.
Therefore the positive integer numbers
$$ \Delta^k \;:=\; \sum_{p+q=k} \left( \dim_\C H^{p,q}_{BC}(X) + \dim_\C H^{p,q}_{A}(X) \right) - 2\, b_k \;\in\; \N \;, $$
varying $k\in\{1,2\}$, measure the non-K\"ahlerness of compact complex surfaces $X$.

\medskip

Compact complex surfaces are divided in seven classes, according to the Kodaira and Enriques classification, see, e.g., \cite{barth-hulek-peters-vandeven}.
In this note, we compute Bott-Chern cohomology for some classes of compact complex (non-K\"ahler) surfaces. In particular, we are interested in studying the relations between Bott-Chern cohomology and de Rham cohomology, looking at the injectivity of the natural map $H^{2,1}_{BC}(X) \to H^{3}_{dR}(X;\C)$. This can be intended as a weak version of the $\del\delbar$-Lemma, compare also \cite{fu-yau}.

More precisely, we start by proving that the non-K\"ahlerness for compact complex surfaces is encoded only in $\Delta^2$, namely, $\Delta^1$ is always zero. This gives a partial answer to a question by T.~C. Dinh to the third author.

\smallskip
\noindent{\bfseries Theorem \ref{thm:delta1=0}.}
{\itshape Let $X$ be a compact complex surface. Then:
 \begin{enumerate}[(i)]
  \item the natural map $H^{2,1}_{BC}(X) \to H^{2,1}_{\delbar}(X)$ induced by the identity is injective;
  \item $\Delta^1 = 0$.
 \end{enumerate}
In particular, the non-K\"ahlerness of $X$ is measured by just $\Delta^2\in\N$.
}
\smallskip

For compact complex surfaces in class $\text{VII}$, we show the following result, where we denote $h^{p,q}_{BC}:=\dim_\C H^{p,q}_{BC}(X)$ for $p,q\in\{0,1,2\}$.

\smallskip
\noindent{\bfseries Theorem \ref{thm:bc-vii}.}
{\itshape The Bott-Chern numbers of compact complex surfaces in class $\text{VII}$ are:
 $$
 \begin{array}{ccccc}
     &  & h^{0,0}_{BC}=1 &  &  \\
     & h^{1,0}_{BC}=0 &  & h^{0,1}_{BC}=0 &  \\
    h^{2,0}_{BC}=0 &  & h^{1,1}_{BC}=b_2+1 &  & h^{0,2}_{BC}=0 \\
     & h^{2,1}_{BC}=1 &  & h^{1,2}_{BC}=1 &  \\
     &  & \phantom{\;.} h^{2,2}_{BC}=1 \;. &  & 
 \end{array}
 $$
}
\smallskip

According to Theorem \ref{thm:delta1=0}, the natural map $H^{2,1}_{BC}(X) \to H^{2,1}_{\delbar}(X)$ is injective for any compact complex surface. One is then interested in studying the injectivity of the natural map $H^{2,1}_{BC}(X) \to H^{3}_{dR}(X;\C)$ induced by the identity, at least for compact complex surfaces diffeomorphic to solvmanifolds. In fact, by definition, the property of satisfying the $\del\delbar$-Lemma, \cite{deligne-griffiths-morgan-sullivan}, is equivalent to the natural map $\bigoplus_{p+q=\bullet} H^{p,q}_{BC}(X) \to H^{\bullet}_{dR}(X;\C)$ being injective. Note that, for a compact complex manifold of complex dimension $n$, the injectivity of the map $H^{n,n-1}_{BC}(X)\to H^{2n-1}_{dR}(X;\C)$ implies the $(n-1,n)$-th weak $\del\delbar$-Lemma in the sense of J. Fu and S.-T. Yau, \cite[Definition 5]{fu-yau}.

We then compute the Bott-Chern cohomology for compact complex surfaces diffeomorphic to solvmanifolds, according to the list given by K. Hasegawa in \cite{hasegawa-jsg}, see Theorem \ref{thm:bc-solvmfds}. More precisely, we prove that the cohomologies can be computed by using just left-invariant forms. Furthermore, for such complex surfaces, we note that the natural map $H^{2,1}_{BC}(X)\to H^{3}_{dR}(X;\C)$ is injective, see Theorem \ref{thm:inj-solvmfds}.

We note that the above classes do not exhaust the set of compact complex non-K\"ahler surfaces, the cohomologies of elliptic surfaces being still unknown.

\medskip

\noindent{\sl Acknowledgments.} The first and third authors would like to thank the Aix-Marseille University for its warm hospitality. Many thanks are due to the referee for her/his suggestions that improved the presentation.
The third author would like to dedicate this paper to the memory of his mother.

\section{Non-K\"ahlerness of compact complex surfaces and Bott-Chern cohomology}

We recall that, for a compact complex manifold of complex dimension $n$, for $k\in\{0,\ldots, 2n\}$, we define the ``non-K\"ahlerness'' degrees, \cite[Theorem A]{angella-tomassini-3},
$$ \Delta^k \;:=\; \sum_{p+q=k} \left( h^{p,q}_{BC} + h^{n-q,n-p}_{BC} \right) - 2\, b_k \in \N \;,. $$
where we use the duality in \cite[\S2.c]{schweitzer} giving $h^{p,q}_{BC}:=\dim_\C H^{p,q}_{BC}(X)=\dim_\C H^{n-q,n-p}_{A}(X)$.
According to \cite[Theorem B]{angella-tomassini-3}, $\Delta^k = 0$ for any $k\in\{0,\ldots,2n\}$ if and only if $X$ satisfies the $\del\delbar$-Lemma, namely, every $\del$-closed $\delbar$-closed $\de$-exact form is $\del\delbar$-exact too. In particular, for a compact complex surface $X$, the condition $\Delta^1=\Delta^2=0$ is equivalent to $X$ being K\"ahler, the first Betti number being even, \cite{kodaira-1, miyaoka, siu}, see also \cite[Corollaire 5.7]{lamari}, and \cite[Theorem 11]{buchdahl}.

\medskip

We prove that $\Delta^1$ is always zero for any compact complex surface. In particular, a sufficient and necessary condition for compact complex surfaces to be K\"ahler is $\Delta^2=0$.

\begin{thm}\label{thm:delta1=0}
 Let $X$ be a compact complex surface. Then:
 \begin{enumerate}[(i)]
  \item the natural map $H^{2,1}_{BC}(X) \to H^{2,1}_{\delbar}(X)$ induced by the identity is injective;
  \item $\Delta^1 = 0$.
 \end{enumerate}
In particular, the non-K\"ahlerness of $X$ is measured by just $\Delta^2\in\N$.
\end{thm}

\begin{proof}
\begin{enumerate}[(i)]
 \item  Let $\alpha\in\wedge^{2,1}X$ be such that $[\alpha]=0\in H^{2,1}_{\delbar}(X)$. Let $\beta\in\wedge^{2,0}X$ be such that $\alpha=\delbar\beta$. Fix a Hermitian metric $g$ on $X$, and consider the Hodge decomposition of $\beta$ with respect to the Dolbeault Laplacian $\overline{\square}$: let $\beta=\beta_h + \delbar^*\lambda$ where $\beta_h\in\wedge^{2,0}X \cap \ker\overline{\square}$, and $\lambda\in\wedge^{2,1}X$. Therefore we have
 $$ \alpha \;=\; \delbar\beta \;=\; \delbar\delbar^*\lambda \;=\; -\delbar* \underbrace{\left(\del*\lambda\right)}_{\in\wedge^{2,0}X} \;=\; -\delbar \left(\del*\lambda\right) \;=\; \del\delbar \left(*\lambda\right) \;, $$
 where we have used that any $(2,0)$-form is primitive and hence, by the Weil identity, is self-dual. In particular, $\alpha$ is $\del\delbar$-exact, so it induces a zero class in $H^{2,1}_{BC}(X)$.

 \item  On the one hand, note that
 \begin{eqnarray*}
  H^{1,0}_{BC}(X) &=& \frac{\ker\del \cap \ker\delbar \cap \wedge^{1,0}X}{\imm\del\delbar} \;=\; \ker\del \cap \ker\delbar \cap \wedge^{1,0}X \\[5pt]
  &\subseteq& \ker\delbar \cap \wedge^{1,0}X \;=\; \frac{\ker\delbar \cap \wedge^{1,0}X}{\imm\delbar} \;=\; H^{1,0}_{\delbar}(X) \;.
 \end{eqnarray*}
 It follows that
 $$ \dim_\C H^{0,1}_{BC}(X) \;=\; \dim_\C H^{1,0}_{BC}(X) \;\leq\; \dim_\C H^{1,0}_{\delbar}(X) \;=\; b_1 - \dim_\C H^{0,1}_{\delbar}(X) \;, $$
 where we use that the Fr\"olicher spectral sequence degenerates, hence in particular $b_1 = \dim_\C H^{1,0}_{\delbar}(X) + \dim_\C H^{0,1}_{\delbar}(X)$.

 On the other hand, by part (i), we have
 $$ \dim_\C H^{1,2}_{BC}(X) \;=\; \dim_\C H^{2,1}_{BC}(X) \;\leq\; \dim_\C H^{2,1}_{\delbar}(X) \;=\; \dim_\C H^{0,1}_{\delbar}(X) \;, $$
 where we use the Kodaira and Serre duality $H^{2,1}_{\delbar}(X) \simeq H^{1}(X;\Omega^2_X) \simeq H^{1}(X;\mathcal{O}_X) \simeq H^{0,1}_{\delbar}(X)$.

 By summing up, we get
 \begin{eqnarray*}
  \Delta^1 &=& \dim_\C H^{0,1}_{BC}(X) + \dim_\C H^{1,0}_{BC}(X) + \dim_\C H^{1,2}_{BC}(X) + \dim_\C H^{2,1}_{BC}(X) - 2\, b_1 \\[5pt]
  &\leq& 2\, \left( b_1 - \dim_\C H^{0,1}_{\delbar}(X) + \dim_\C H^{0,1}_{\delbar}(X) - b_1 \right) \;=\; 0 \;,
 \end{eqnarray*}
 concluding the proof.\qedhere
\end{enumerate}
\end{proof}

\section{Class VII surfaces}
In this section, we compute Bott-Chern cohomology for compact complex surfaces in class $\text{VII}$.

\medskip

Let $X$ be a compact complex surface.
By Theorem \ref{thm:delta1=0}, the natural map $H^{2,1}_{BC}(X) \to H^{2,1}_{\delbar}(X)$ is always injective. Consider now the case when $X$ is in class $\text{VII}$. If $X$ is minimal, we prove that the same holds for cohomology with values in a line bundle. We will also prove that the natural map $H^{1,2}_{BC}(X) \to H^{1,2}_{\delbar}(X)$ is not injective.

\begin{prop}
 Let $X$ be a compact complex surface in class $\text{VII}_0$. Let $L \in H^1(X;\C^*)=\mathrm{Pic}^0(X)$. The natural map $H^{2,1}_{BC}(X;L) \to H^{2,1}_{\delbar}(X;L)$ induced by the identity is injective.
\end{prop}

\begin{proof}
 Let $\alpha\in\wedge^{2,1}X\otimes L$ be a $\delbar_L$-exact $(2,1)$-form. We need to prove that $\alpha$ is $\del_L\delbar_L$-exact too. Consider $\alpha = \delbar_L \vartheta$, where $\vartheta \in \wedge^{2,0}X\otimes L$. In particular, $\del_L\vartheta=0$, hence $\bar\vartheta$ defines a class in $H^{0,2}_{\delbar}(X;L)$. Note that $H^{0,2}_{\delbar}(X;L) \simeq H^2(X;\mathcal{O}_X(L)) \simeq H^{0}(X;K_X\otimes L^{-1}) = \{0\}$ for surfaces of class $\text{VII}_0$, \cite[Remark 2.21]{dloussky-ajm}. It follows that $\bar\vartheta = - \delbar_L \bar\eta$ for some $\eta\in\wedge^{1,0}X\otimes L$. Hence $\alpha=\del_L\delbar_L\eta$, that is, $\alpha$ is $\del_L\delbar_L$-exact.
\end{proof}

\medskip

We now compute the Bott-Chern cohomology of class $\text{VII}$ surfaces.

\begin{thm}\label{thm:bc-vii}
 The Bott-Chern numbers of compact complex surfaces in class $\text{VII}$ are:
 $$
 \begin{array}{ccccc}
     &  & h^{0,0}_{BC}=1 &  &  \\
     & h^{1,0}_{BC}=0 &  & h^{0,1}_{BC}=0 &  \\
    h^{2,0}_{BC}=0 &  & h^{1,1}_{BC}=b_2+1 &  & h^{0,2}_{BC}=0 \\
     & h^{2,1}_{BC}=1 &  & h^{1,2}_{BC}=1 &  \\
     &  & \phantom{\;.} h^{2,2}_{BC}=1 \;. &  & 
 \end{array}
 $$
\end{thm}

\begin{proof}
 It holds $H^{1,0}_{BC}(X)=\frac{\ker\del\cap\ker\delbar \cap \wedge^{1,0}X}{\imm\del\delbar}=\ker\del\cap\ker\delbar \cap \wedge^{1,0}X \subseteq \ker\delbar \cap \wedge^{1,0}X = \frac{\ker\delbar \cap \wedge^{1,0}X}{\imm\delbar} = H^{1,0}_{\delbar}(X) = \{0\}$ hence $h^{1,0}_{BC}=h^{0,1}_{BC}=0$.

 On the other side, by Theorem \ref{thm:delta1=0}, $0=\Delta^1=2\,\left(h^{1,0}_{BC}+h^{2,1}_{BC}-b_1\right)=2\, \left(h^{2,1}_{BC}-1\right)$ hence $h^{2,1}_{BC}=h^{1,2}_{BC}=1$.

 Similarly, it holds $H^{2,0}_{BC}(X)=\frac{\ker\del\cap\ker\delbar \cap \wedge^{2,0}X}{\imm\del\delbar}=\ker\del\cap\ker\delbar \cap \wedge^{2,0}X \subseteq \ker\delbar \cap \wedge^{2,0}X = \frac{\ker\delbar \cap \wedge^{2,0}X}{\imm\delbar} = H^{2,0}_{\delbar}(X) = \{0\}$ hence $h^{2,0}_{BC}=h^{0,2}_{BC}=0$.

 Note that, from \cite[Theorem A]{angella-tomassini-3}, we have $0\leq\Delta^2=2\,\left(h^{2,0}_{BC}+h^{1,1}_{BC}+h^{0,2}_{BC}-b_2\right)=2\,\left(h^{1,1}_{BC}-b_2\right)$ hence $h^{1,1}_{BC} \geq b_2$. More precisely, from \cite[Theorem B]{angella-tomassini-3} and Theorem \ref{thm:delta1=0}, we have that $h^{1,1}_{BC} = b_2$ if and only if $\Delta^2=0$ if and only if $X$ satisfies the $\del\delbar$-Lemma, in fact $X$ is K\"ahler, which is not the case.

 Finally, we prove that $h^{1,1}_{BC}=b_2+1$. Consider the following exact sequences from \cite[Lemma 2.3]{teleman}. More precisely, the sequence
 $$ 0 \to \frac{\imm\de \cap \wedge^{1,1}X}{\imm\del\delbar} \to H^{1,1}_{BC}(X) \to \imm \left( H^{1,1}_{BC}(X) \to H^{2}_{dR}(X;\C) \right) \to 0 $$
 is clearly exact.
 Furthermore, fix a Gauduchon metric $g$. Denote by $\omega:=g(J\sspace, \ssspace)$ the $(1,1)$-form associated to $g$, where $J$ denotes the integrable almost-complex structure. By definition of $g$ being Gauduchon, we have $\del\delbar\omega=0$. The sequence
 $$ 0 \to \frac{\imm\de \cap \wedge^{1,1}X}{\imm\del\delbar} \stackrel{\left\langle \sspace \middle\vert \omega \right\rangle}{\to} \C $$
 is exact. Indeed, firstly note that for $\eta=\del\delbar f\in\imm\del\delbar\cap\wedge^{1,1}X$, we have
 $$ \left\langle \eta \middle\vert \omega \right\rangle \;=\; \int_X \del\delbar f \wedge \overline{*\omega} \;=\; \int_X \del\delbar f \wedge \omega \;=\; \int_X f\, \del\delbar\omega \;=\; 0 $$
 by applying twice the Stokes theorem.
 Then, we recall the argument in \cite[Lemma 2.3(ii)]{teleman} for proving that the map
 $$ \left\langle \sspace \middle\vert \omega \right\rangle \colon \frac{\imm\de \cap \wedge^{1,1}X}{\imm\del\delbar} \to \C $$
 is injective. Take $\alpha=\de\beta\in\imm\de\cap\wedge^{1,1}X\cap\ker \left\langle \sspace \middle\vert \omega \right\rangle$. Then
 $$ \left\langle \Lambda\alpha \middle\vert 1 \right\rangle \;=\; \left\langle \alpha \middle\vert \omega \right\rangle \;=\; 0 \;, $$
 where $\Lambda$ is the adjoint operator of $\omega\wedge\sspace$ with respect to $\left\langle \sspace \middle\vert \ssspace \right\rangle$. Then $\Lambda \alpha \in \ker \left\langle \sspace \middle\vert 1 \right\rangle = \imm \Lambda\del\delbar$, by extending \cite[Corollary 7.2.9]{lubke-teleman} by $\C$-linearity. Take $u\in\mathcal{C}^\infty(X;\C)$ such that $\Lambda\alpha=\Lambda\del\delbar u$. Then, by defining $\alpha':=\alpha-\del\delbar u$, we have $\left[\alpha'\right]=\left[\alpha\right]\in \frac{\imm\de \cap \wedge^{1,1}X}{\imm\del\delbar}$, and $\Lambda \alpha'=0$, and $\alpha'=\de\beta'$ where $\beta':=\beta-\delbar u$.
 In particular, $\alpha'$ is primitive. Since $\alpha'$ is primitive and of type $(1,1)$, then it is anti-self-dual by the Weil identity. Then
 $$ \left\| \alpha' \right\|^2 \;=\; \left\langle \alpha' \middle\vert \alpha' \right\rangle \;=\; \int_X \alpha' \wedge \overline{*\alpha'} \;=\; - \int_X \alpha' \wedge \overline{\alpha'} \;=\; - \int_X \de\beta' \wedge \de\overline{\beta'} \;=\; - \int_X \de \left( \beta' \wedge \de \overline{\beta'} \right) \;=\; 0 $$
 and hence $\alpha'=0$, and therefore $\left[\alpha\right]=0$.

 Since the space $\frac{\imm\de \cap \wedge^{1,1}X}{\imm\del\delbar}$ is finite-dimensional, being a sub-space of $H^{1,1}_{BC}(X)$, and since the space $\imm \left( H^{1,1}_{BC}(X) \to H^{2}_{dR}(X;\C) \right)$ is finite-dimensional, being a sub-space of $H^{2}_{dR}(X;\C)$, we get that
 $$ \dim_\C \frac{\imm\de \cap \wedge^{1,1}X}{\imm\del\delbar} \;\leq\; \dim_\C \C \;=\; 1 \;, $$
 and hence
 $$ b_2 \;<\; \dim_\C H^{1,1}_{BC}(X) \;=\; \dim_\C \imm \left( H^{1,1}_{BC}(X) \to H^{2}_{dR}(X;\C) \right) + \dim_\C \frac{\imm\de \cap \wedge^{1,1}X}{\imm\del\delbar} \;\leq\; b_2 + 1 \;. $$
 We get that $\dim_\C H^{1,1}_{BC}(X)=b_2+1$.
\end{proof}

\medskip

Finally, we prove that the natural map $H^{1,2}_{BC}(X) \to H^{1,2}_{\delbar}(X)$ is not injective.

\begin{prop}
 Let $X$ be a compact complex surface in class $\text{VII}$. Then the natural map $H^{1,2}_{BC}(X) \to H^{1,2}_{\delbar}(X)$ induced by the identity is the zero map and not an isomorphism.
\end{prop}

\begin{proof}
 Note that, for class $\text{VII}$ surfaces, the pluri-genera are zero. In particular, $H^{1,2}_{\delbar}(X) \simeq H^{1,0}_{\delbar}(X)=\{0\}$, by Kodaira and Serre duality. By Theorem \ref{thm:bc-vii}, one has $H^{1,2}_{BC}(X)\neq\{0\}$.
\end{proof}

\subsection{Cohomologies of Calabi-Eckmann surface}
In this section, as an explicit example, we list the representatives of the cohomologies of a compact complex surface in class $\text{VII}$: namely, we consider the Calabi-Eckmann structure on the differentiable manifolds underlying the Hopf surfaces.

\medskip

 Consider the differentiable manifold $X:=\mathbb{S}^1\times\mathbb{S}^3$. As a Lie group, $\mathbb{S}^3=SU(2)$ has a global left-invariant co-frame $\left\{e^1,e^2,e^3\right\}$ such that $\de e^1=-2e^2\wedge e^3$ and $\de e^2=2e^1\wedge e^3$ and $\de e^3=-2 e^1\wedge e^2$. Hence, we consider a global left-invariant co-frame $\left\{f, e^1,e^2,e^3\right\}$ on $X$ with structure equations
 $$ \left\{\begin{array}{rcl}
            \de f   &=& \phantom{+} 0 \\[5pt]
            \de e^1 &=&          -  2\, e^2 \wedge e^3 \\[5pt]
            \de e^2 &=& \phantom{+} 2\, e^1 \wedge e^3 \\[5pt]
            \de e^3 &=&          -  2\, e^1 \wedge e^2
           \end{array}\right. \;.
 $$

 Consider the left-invariant almost-complex structure defined by the $(1,0)$-forms
 $$ \left\{\begin{array}{rcl}
            \varphi^1 &:=& e^1 + \im\, e^2 \\[5pt]
            \varphi^2 &:=& e^3 + \im\, f
           \end{array}\right. \;.
 $$
 By computing the complex structure equations, we get
 $$ \left\{\begin{array}{rcl}
            \del \varphi^1 &=& \im\, \varphi^1 \wedge \varphi^2 \\[5pt]
            \del \varphi^2 &=& 0
           \end{array}\right.
    \qquad \text{ and } \qquad
    \left\{\begin{array}{rcl}
            \delbar \varphi^1 &=& \im\, \varphi^1 \wedge \bar\varphi^2 \\[5pt]
            \delbar \varphi^2 &=& - \im\, \varphi^1 \wedge \bar\varphi^1
           \end{array}\right. \;.
 $$
 We note that the almost-complex structure is in fact integrable.

 The manifold $X$ is a compact complex manifold not admitting K\"ahler metrics. It is bi-holomorphic to the complex manifold $M_{0,1}$ considered by Calabi and Eckmann, \cite{calabi-eckmann}, see \cite[Theorem 4.1]{parton-rend}.

 Consider the Hermitian metric $g$ whose associated $(1,1)$-form is
 $$ \omega \;:=\; \frac{\im}{2}\, \sum_{j=1}^{2} \varphi^j \wedge \bar\varphi^j \;. $$

 As for the de Rham cohomology, from the K\"unneth formula we get
 $$ H^\bullet_{dR}(X;\C) \;=\; \C \left\langle 1 \right\rangle \oplus \C\left\langle \varphi^{2}-\bar\varphi^{2} \right\rangle \oplus \C\left\langle \varphi^{12\bar1}-\varphi^{1\bar1\bar2} \right\rangle \oplus \C\left\langle \varphi^{12\bar1\bar2} \right\rangle \;, $$
 (where, here and hereafter, we shorten, e.g., $\varphi^{12\bar1} := \varphi^1\wedge\varphi^2\wedge\bar\varphi^1$).

 By \cite[Appendix II, Theorem 9.5]{hirzebruch}, one has that a model for the Dolbeault cohomology is given by
 $$ H_{\delbar}^{\bullet,\bullet}(X) \;\simeq\; \bigwedge \left\langle x_{2,1},\, x_{0,1} \right\rangle \;, $$
 where $x_{i,j}$ is an element of bi-degree $(i,j)$. In particular, we recover that the Hodge numbers $\left\{ h^{p,q}_{\delbar}:=\dim_\C H^{p,q}_{\delbar}(X) \right\}_{p,q\in\{0,1,2\}}$ are
 $$
 \begin{array}{ccccc}
   &  & h^{0,0}_{\delbar}=1 &  & \\
   & h^{1,0}_{\delbar}=0 &  & h^{0,1}_{\delbar}=1 & \\
  h^{2,0}_{\delbar}=0 &  & h^{1,1}_{\delbar}=0 &  & h^{0,2}_{\delbar}=0 \\
   & h^{2,1}_{\delbar}=1 &  & h^{1,2}_{\delbar}=0 & \\
   &  & h^{2,2}_{\delbar}=1 &  & \\
 \end{array} \;.
 $$
 We note that the sub-complex
 $$ \iota \colon \bigwedge \left\langle \varphi^1,\, \varphi^2,\, \bar\varphi^1,\, \bar\varphi^2 \right\rangle \hookrightarrow \wedge^{\bullet,\bullet}X $$
 is such that $H_{\delbar}(\iota)$ is an isomorphism. More precisely, we get
 \begin{eqnarray*}
  H^{\bullet,\bullet}_{\delbar}(X) &=& \C\left\langle 1 \right\rangle \oplus \C\left\langle \left[\varphi^{\bar{2}}\right] \right\rangle \oplus \C\left\langle \left[ \varphi^{12\bar1} \right] \right\rangle \oplus \C\left\langle \left[ \varphi^{12\bar1\bar2}\right] \right\rangle \;,
 \end{eqnarray*}
 where we have listed the harmonic representatives with respect to the Dolbeault Laplacian of $g$.
 
 By \cite[Theorem 1.3, Proposition 2.2]{angella-kasuya-1}, we have also $H_{BC}(\iota)$ isomorphism. In particular, we get
 \begin{eqnarray*}
  H^{\bullet,\bullet}_{BC}(X) &=& \C\left\langle 1 \right\rangle \oplus \C\left\langle \left[\varphi^{1\bar1}\right] \right\rangle \oplus \C\left\langle \left[\varphi^{12\bar1}\right] \right\rangle \oplus \C\left\langle \left[\varphi^{1\bar1\bar2}\right] \right\rangle \oplus \C\left\langle \left[\varphi^{12\bar1\bar2}\right] \right\rangle \;,
 \end{eqnarray*}
 where we have listed the harmonic representatives with respect to the Bott-Chern Laplacian of $g$.

 By \cite[\S2.c]{schweitzer}, we have
 \begin{eqnarray*}
  H^{\bullet,\bullet}_{A}(X) &=& \C\left\langle 1 \right\rangle \oplus \C\left\langle \left[\varphi^{2}\right] \right\rangle \oplus \C\left\langle \left[\varphi^{\bar2}\right] \right\rangle \oplus \C\left\langle \left[\varphi^{2\bar2}\right] \right\rangle \oplus \C\left\langle \left[\varphi^{12\bar1\bar2}\right] \right\rangle \;,
 \end{eqnarray*}
 where we have listed the harmonic representatives with respect to the Aeppli Laplacian of $g$.

 Note in particular that the natural map $H^{2,1}_{BC}(X) \to H^{2,1}_{\delbar}(X)$ induced by the identity is an isomorphism, and that the natural map $H^{2,1}_{BC}(X) \to H^3_{dR}(X;\C)$ induced by the identity is injective.

\section{Complex surfaces diffeomorphic to solvmanifolds}

Let $X$ be a compact complex surface diffeomorphic to a solvmanifold $\left. \Gamma \middle\backslash G \right.$. By \cite[Theorem 1]{hasegawa-jsg}, $X$ is
\begin{inparaenum}[\itshape (A)]
 \item either a complex torus,
 \item or a hyperelliptic surface,
 \item or a Inoue surface of type $\mathcal{S}_M$,
 \item or a primary Kodaira surface,
 \item or a secondary Kodaira surface,
 \item or a Inoue surface of type $\mathcal{S}^\pm$,
\end{inparaenum}
and, as such, it is endowed with a left-invariant complex structure.

\medskip

In each case, we recall the structure equations of the group $G$, see \cite{hasegawa-jsg}. More precisely, take a basis $\left\{ e_1, e_2, e_3, e_4 \right\}$ of the Lie algebra $\mathfrak{g}$ naturally associated to $G$. We have the following commutation relations, according to \cite{hasegawa-jsg}:
\begin{enumerate}[(A)]
 \item differentiable structure underlying a {\em complex torus}:
  $$ \left[e_j,e_k\right]=0 \qquad \text{ for any } j,k\in\{1,2,3,4\} \;; $$
  (hereafter, we write only the non-trivial commutators);
 \item differentiable structure underlying a {\em hyperelliptic surface}:
  $$ \left[e_1,e_4\right]=e_2 \;, \quad \left[e_2,e_4\right]=-e_1 \;; $$
 \item differentiable structure underlying a {\em Inoue surface of type $\mathcal{S}_M$}:
  $$ \left[e_1,e_4\right]=-\alpha\, e_1+\beta\, e_2 \;, \quad \left[e_2,e_4\right]=-\beta\,e_1-\alpha\,e_2 \;, \quad \left[e_3,e_4\right]=2\alpha\,e_3 \;, $$
  where $\alpha\in\R\setminus\{0\}$ and $\beta\in\R$;
 \item differentiable structure underlying a {\em primary Kodaira surface}:
  $$ \left[e_1,e_2\right]=-e_3 \;; $$
 \item differentiable structure underlying a {\em secondary Kodaira surface}:
  $$ \left[e_1,e_2\right]=-e_3 \;, \quad \left[e_1,e_4\right]=e_2 \;, \quad \left[e_2,e_4\right]=-e_1 \;; $$
 \item differentiable structure underlying a {\em Inoue surface of type $\mathcal{S}^\pm$}:
  $$ \left[e_2,e_3\right]=-e_1 \;, \quad \left[e_2,e_4\right]=-e_2 \;, \quad \left[e_3,e_4\right]=e_3 \;. $$
\end{enumerate}

Denote by $\left\{ e^1, e^2, e^3, e^4 \right\}$ the dual basis of $\left\{e_1, e_2, e_3, e_4\right\}$. We recall that, for any $\alpha\in\mathfrak{g}^*$, for any $x,y\in\mathfrak{g}$, it holds $\de \alpha (x,y)=-\alpha\left([x,y]\right)$. Hence we get the following structure equations:
\begin{enumerate}[(A)]
 \item differentiable structure underlying a {\em complex torus}:
  $$ \left\{ \begin{array}{lcl}
              \de e^1 &=& 0 \\[5pt]
              \de e^2 &=& 0 \\[5pt]
              \de e^3 &=& 0 \\[5pt]
              \de e^4 &=& 0
             \end{array} \right. \;; $$
 \item differentiable structure underlying a {\em hyperelliptic surface}:
  $$ \left\{ \begin{array}{lcl}
              \de e^1 &=& e^2 \wedge e^4 \\[5pt]
              \de e^2 &=& - e^1 \wedge e^4 \\[5pt]
              \de e^3 &=& 0 \\[5pt]
              \de e^4 &=& 0
             \end{array} \right. \;; $$
 \item differentiable structure underlying a {\em Inoue surface of type $\mathcal{S}_M$}:
  $$ \left\{ \begin{array}{lcl}
              \de e^1 &=& \alpha\, e^1\wedge e^4 + \beta\, e^2\wedge e^4 \\[5pt]
              \de e^2 &=& -\beta\, e^1\wedge e^4 + \alpha\, e^2\wedge e^4 \\[5pt]
              \de e^3 &=& -2\alpha\, e^3\wedge e^4 \\[5pt]
              \de e^4 &=& 0
             \end{array} \right. \;; $$
 \item differentiable structure underlying a {\em primary Kodaira surface}:
  $$ \left\{ \begin{array}{lcl}
              \de e^1 &=& 0 \\[5pt]
              \de e^2 &=& 0 \\[5pt]
              \de e^3 &=& e^1 \wedge e^2 \\[5pt]
              \de e^4 &=& 0
             \end{array} \right. \;; $$
 \item differentiable structure underlying a {\em secondary Kodaira surface}:
  $$ \left\{ \begin{array}{lcl}
              \de e^1 &=& e^2 \wedge e^4 \\[5pt]
              \de e^2 &=& - e^1 \wedge e^4 \\[5pt]
              \de e^3 &=& e^1 \wedge e^2 \\[5pt]
              \de e^4 &=& 0
             \end{array} \right. \;; $$
 \item differentiable structure underlying a {\em Inoue surface of type $\mathcal{S}^\pm$}:
  $$ \left\{ \begin{array}{lcl}
              \de e^1 &=& e^2 \wedge e^3 \\[5pt]
              \de e^2 &=& e^2 \wedge e^4 \\[5pt]
              \de e^3 &=& -e^3\wedge e^4 \\[5pt]
              \de e^4 &=& 0
             \end{array} \right. \;. $$
\end{enumerate}

In cases (A), (B), (C), (D), (E), consider the $G$-left-invariant almost-complex structure $J$ on $X$ defined by
$$ J e_1 \;:=\; e_2 \quad \text{ and } J e_2 \;:=\; -e_1 \quad \text{ and } \quad J e_3 \;:=\; e_4 \quad \text{ and } J e_4 \;:=\; -e_3 \;. $$
Consider the $G$-left-invariant $(1,0)$-forms
$$ \left\{ \begin{array}{l}
            \varphi^1 \;:=\; e^1 + \im e^2 \\[5pt]
            \varphi^2 \;:=\; e^3 + \im e^4
           \end{array}
 \right. \;.
$$

In case (F), consider the $G$-left-invariant almost-complex structure $J$ on $X$ defined by
$$ J e_1 \;:=\; e_2 \quad \text{ and } J e_2 \;:=\; -e_1 \quad \text{ and } \quad J e_3 \;:=\; e_4 - q\, e_2 \quad \text{ and } J e_4 \;:=\; -e_3 - q\, e_1 \;, $$
where $q \in \R$.
Consider the $G$-left-invariant $(1,0)$-forms
$$ \left\{ \begin{array}{l}
            \varphi^1 \;:=\; e^1 + \im e^2 + \im\,q\, e^4 \\[5pt]
            \varphi^2 \;:=\; e^3 + \im e^4
           \end{array}
 \right. \;.
$$

With respect to the $G$-left-invariant coframe $\left\{\varphi^1,\, \varphi^2\right\}$ for the holomorphic tangent bundle $T^{1,0}\left. \Gamma \middle\backslash G \right.$, we have the following structure equations. (As for notation, we shorten, e.g., $\varphi^{1\bar2} := \varphi^{1} \wedge \bar\varphi^{2}$.)

\begin{enumerate}[(A)]
 \item {\em torus}:
  $$
   \left\{
    \begin{array}{lcl}
     \de\varphi^1 & = & 0 \\[5pt]
     \de\varphi^2 & = & 0
    \end{array}
   \right.
  $$

 \item {\em hyperelliptic surface}:
  $$
   \left\{
    \begin{array}{lcl}
     \de\varphi^1 & = & -\frac{1}{2}\, \varphi^{12} + \frac{1}{2}\, \varphi^{1\bar2} \\[5pt]
     \de\varphi^2 & = & 0
    \end{array}
   \right.
  $$

 \item {\em Inoue surface $\mathcal{S}_M$}:
  $$
   \left\{
    \begin{array}{lcl}
     \de\varphi^1 & = & \frac{\alpha-\im\beta}{2\im}\, \varphi^{12} - \frac{\alpha-\im\beta}{2\im}\, \varphi^{1\bar2} \\[5pt]
     \de\varphi^2 & = & -\im\alpha\, \varphi^{2\bar2}
    \end{array}
   \right.
  $$
 (where $\alpha\in\R\setminus\{0\}$ and $\beta\in\R$);

 \item {\em primary Kodaira surface}:
  $$
   \left\{
    \begin{array}{lcl}
     \de\varphi^1 & = & 0 \\[5pt]
     \de\varphi^2 & = & \frac{\im}{2}\, \varphi^{1\bar1}
    \end{array}
   \right.
  $$

 \item {\em secondary Kodaira surface}:
  $$
   \left\{
    \begin{array}{lcl}
     \de\varphi^1 & = & -\frac{1}{2}\, \varphi^{12} + \frac{1}{2}\, \varphi^{1\bar2} \\[5pt]
     \de\varphi^2 & = & \frac{\im}{2}\, \varphi^{1\bar1}
    \end{array}
   \right.
  $$

 \item {\em Inoue surface $\mathcal{S}^\pm$}:
  $$
   \left\{
    \begin{array}{lcl}
     \de\varphi^1 & = & \frac{1}{2\im}\, \varphi^{12}+\frac{1}{2\im}\, \varphi^{2\bar1} + \frac{q\,\im}{2}\, \varphi^{2\bar2} \\[5pt]
     \de\varphi^2 & = & \frac{1}{2\im}\, \varphi^{2\bar2}
    \end{array}
   \right. \;.
  $$
\end{enumerate}

\section{Cohomologies of complex surfaces diffeomorphic to solvmanifolds}

In this section, we compute the Dolbeault and Bott-Chern cohomologies of the compact complex surfaces diffeomorphic to a solvmanifold.

\medskip

We prove the following theorem.

\begin{thm}\label{thm:bc-solvmfds}
 Let $X$ be a compact complex surface diffeomorphic to a solvmanifold $\left. \Gamma \middle\backslash G \right.$; denote the Lie algebra of $G$ by $\mathfrak{g}$. Then the inclusion $\left(\wedge^{\bullet,\bullet}\mathfrak{g}^*,\, \del,\, \delbar\right) \hookrightarrow \left(\wedge^{\bullet,\bullet}X,\, \del,\, \delbar\right)$ induces an isomorphism both in Dolbeault and in Bott-Chern cohomologies. In particular, the dimensions of the de Rham, Dolbeault, and Bott-Chern cohomologies and the degrees of non-K\"ahlerness are summarized in Table \ref{table:cohom}.
\end{thm}

\begin{proof}
 Firstly, we compute the cohomologies of the sub-complex of $G$-left-invariant forms. The computations are straightforward from the structure equations.

\begin{center}
\begin{table}[ht]
 \centering
\resizebox{\textwidth}{!}{
\begin{tabular}{>{$\mathbf\bgroup}l<{\mathbf\egroup$} || >{$}l<{$} | >{$}c<{$} | >{$}l<{$} | >{$}c<{$} || >{$}l<{$} | >{$}c<{$} | >{$}l<{$} | >{$}c<{$} || >{$}l<{$} | >{$}c<{$} | >{$}l<{$} | >{$}c<{$} ||}
\toprule
 & \multicolumn{4}{c||}{\text{(A) torus}} & \multicolumn{4}{c||}{\text{(B) hyperelliptic}} & \multicolumn{4}{c||}{\text{(C) Inoue }$\mathcal{S}_M$} \\
(p,q) & H^{p,q}_{\delbar} & \dim_\C H^{p,q}_{\delbar}  & H^{p,q}_{BC} & \dim_\C H^{p,q}_{BC}  & H^{p,q}_{\delbar} & \dim_\C H^{p,q}_{\delbar}  & H^{p,q}_{BC} & \dim_\C H^{p,q}_{BC} & H^{p,q}_{\delbar} & \dim_\C H^{p,q}_{\delbar}  & H^{p,q}_{BC} & \dim_\C H^{p,q}_{BC} \\
\toprule
(0,0) & \left\langle 1 \right\rangle & 1 & \left\langle 1 \right\rangle & 1 & \left\langle 1 \right\rangle & 1 & \left\langle 1 \right\rangle & 1 & \left\langle 1 \right\rangle & 1 & \left\langle 1 \right\rangle & 1 \\
\midrule[0.02em]
(1,0) & \left\langle \varphi^{1},\, \varphi^{2} \right\rangle & 2 & \left\langle \varphi^{1},\, \varphi^{2} \right\rangle & 2 & \left\langle \varphi^2 \right\rangle & 1 & \left\langle \varphi^2 \right\rangle & 1 & \left\langle 0 \right\rangle & 0 & \left\langle 0 \right\rangle & 0 \\
(0,1) & \left\langle \varphi^{\bar1},\, \varphi^{\bar2} \right\rangle & 2 & \left\langle \varphi^{\bar1},\, \varphi^{\bar2} \right\rangle & 2 & \left\langle \varphi^{\bar2} \right\rangle & 1 & \left\langle \varphi^{\bar2} \right\rangle & 1 & \left\langle \varphi^{\bar2} \right\rangle & 1 & \left\langle 0 \right\rangle & 0\\
\midrule[0.02em]
(2,0) & \left\langle \varphi^{12} \right\rangle & 1 & \left\langle \varphi^{12} \right\rangle & 1 & \left\langle 0 \right\rangle & 0& \left\langle 0 \right\rangle & 0 & \left\langle 0 \right\rangle & 0 & \left\langle 0 \right\rangle & 0 \\
(1,1) & \left\langle \varphi^{1\bar1},\, \varphi^{1\bar2},\, \varphi^{2\bar1},\, \varphi^{2\bar2} \right\rangle & 4 & \left\langle \varphi^{1\bar1},\, \varphi^{1\bar2},\, \varphi^{2\bar1},\, \varphi^{2\bar2} \right\rangle & 4 & \left\langle \varphi^{1\bar1},\, \varphi^{2\bar2} \right\rangle & 2 & \left\langle \varphi^{1\bar1},\, \varphi^{2\bar2} \right\rangle & 2 & \left\langle 0 \right\rangle & 0 & \left\langle \varphi^{2\bar2} \right\rangle & 1 \\
(0,2) & \left\langle \varphi^{\bar1\bar2} \right\rangle & 1 & \left\langle \varphi^{\bar1\bar2} \right\rangle & 1 & \left\langle 0 \right\rangle & 0 & \left\langle 0 \right\rangle & 0 & \left\langle 0 \right\rangle & 0 & \left\langle 0 \right\rangle & 0 \\
\midrule[0.02em]
(2,1) & \left\langle \varphi^{12\bar1},\, \varphi^{12\bar2} \right\rangle & 2  & \left\langle \varphi^{12\bar1},\, \varphi^{12\bar2} \right\rangle & 2 & \left\langle \varphi^{12\bar1} \right\rangle & 1 & \left\langle \varphi^{12\bar1} \right\rangle & 1 & \left\langle \varphi^{12\bar1} \right\rangle & 1 & \left\langle \varphi^{12\bar1} \right\rangle & 1 \\
(1,2) & \left\langle \varphi^{1\bar1\bar2},\, \varphi^{2\bar1\bar2} \right\rangle & 2 & \left\langle \varphi^{1\bar1\bar2},\, \varphi^{2\bar1\bar2} \right\rangle & 2 & \left\langle \varphi^{1\bar1\bar2} \right\rangle & 1 & \left\langle \varphi^{1\bar1\bar2} \right\rangle & 1 & \left\langle 0 \right\rangle & 0 & \left\langle \varphi^{1\bar1\bar2} \right\rangle & 1 \\
\midrule[0.02em]
(2,2) & \left\langle \varphi^{12\bar1\bar2} \right\rangle & 1 & \left\langle \varphi^{12\bar1\bar2} \right\rangle & 1 & \left\langle \varphi^{12\bar1\bar2} \right\rangle & 1 & \left\langle \varphi^{12\bar1\bar2} \right\rangle & 1 & \left\langle \varphi^{12\bar1\bar2} \right\rangle & 1 & \left\langle \varphi^{12\bar1\bar2} \right\rangle & 1 \\
\bottomrule
\end{tabular}
}
\caption{Dolbeault and Bott-Chern cohomologies of compact complex surfaces diffeomorphic to solvmanifolds, part 1.}
\label{table:cohomologies-1}
\end{table}
\end{center}

\begin{center}
\begin{table}[ht]
 \centering
\resizebox{\textwidth}{!}{
\begin{tabular}{>{$\mathbf\bgroup}l<{\mathbf\egroup$} || >{$}l<{$} | >{$}c<{$} | >{$}l<{$} | >{$}c<{$} || >{$}l<{$} | >{$}c<{$} | >{$}l<{$} | >{$}c<{$} || >{$}l<{$} | >{$}c<{$} | >{$}l<{$} | >{$}c<{$} ||}
\toprule
 & \multicolumn{4}{c||}{\text{(D) primary Kodaira}} & \multicolumn{4}{c||}{\text{(E) secondary Kodaira}} & \multicolumn{4}{c||}{\text{(F) Inoue }$\mathcal{S}_\pm$} \\
(p,q) & H^{p,q}_{\delbar} & \dim_\C H^{p,q}_{\delbar}  & H^{p,q}_{BC} & \dim_\C H^{p,q}_{BC}  & H^{p,q}_{\delbar} & \dim_\C H^{p,q}_{\delbar}  & H^{p,q}_{BC} & \dim_\C H^{p,q}_{BC} & H^{p,q}_{\delbar} & \dim_\C H^{p,q}_{\delbar}  & H^{p,q}_{BC} & \dim_\C H^{p,q}_{BC} \\
\toprule
(0,0) & \left\langle 1 \right\rangle & 1 & \left\langle 1 \right\rangle & 1 & \left\langle 1 \right\rangle & 1 & \left\langle 1 \right\rangle & 1 & \left\langle 1 \right\rangle & 1 & \left\langle 1 \right\rangle & 1 \\
\midrule[0.02em]
(1,0) & \left\langle \varphi^{1} \right\rangle & 1 & \left\langle \varphi^{1} \right\rangle & 1 & \left\langle 0 \right\rangle & 0 & \left\langle 0 \right\rangle & 0 & \left\langle 0 \right\rangle & 0 & \left\langle 0 \right\rangle & 0 \\
(0,1) & \left\langle \varphi^{\bar1},\, \varphi^{\bar2} \right\rangle & 2 & \left\langle \varphi^{\bar1} \right\rangle & 1 & \left\langle \varphi^{\bar2} \right\rangle & 1 & \left\langle 0 \right\rangle & 0 & \left\langle \varphi^{\bar2} \right\rangle & 1 & \left\langle 0 \right\rangle & 0 \\
\midrule[0.02em]
(2,0) & \left\langle \varphi^{12} \right\rangle & 1 & \left\langle \varphi^{12} \right\rangle & 1 & \left\langle 0 \right\rangle & 0 & \left\langle 0 \right\rangle & 0 & \left\langle 0 \right\rangle & 0 & \left\langle 0 \right\rangle & 0 \\
(1,1) & \left\langle \varphi^{1\bar2},\, \varphi^{2\bar1} \right\rangle & 2 & \left\langle \varphi^{1\bar1},\, \varphi^{1\bar2},\, \varphi^{2\bar1} \right\rangle & 3 & \left\langle 0 \right\rangle & 0 & \left\langle \varphi^{1\bar1} \right\rangle & 1 & \left\langle 0 \right\rangle & 0 & \left\langle \varphi^{2\bar2} \right\rangle & 1 \\
(0,2) & \left\langle \varphi^{\bar1\bar2} \right\rangle & 1 & \left\langle \varphi^{\bar1\bar2} \right\rangle & 1 & \left\langle 0 \right\rangle & 0 & \left\langle 0 \right\rangle & 0 & \left\langle 0 \right\rangle & 0 & \left\langle 0 \right\rangle & 0 \\
\midrule[0.02em]
(2,1) & \left\langle \varphi^{12\bar1},\, \varphi^{12\bar2} \right\rangle & 2 & \left\langle \varphi^{12\bar1},\, \varphi^{12\bar2} \right\rangle & 2 & \left\langle \varphi^{12\bar1} \right\rangle & 1 & \left\langle \varphi^{12\bar1} \right\rangle & 1 & \left\langle \varphi^{12\bar1} \right\rangle & 1 & \left\langle \varphi^{12\bar1} \right\rangle & 1 \\
(1,2) & \left\langle \varphi^{2\bar1\bar2} \right\rangle & 1 & \left\langle \varphi^{1\bar1\bar2},\, \varphi^{2\bar1\bar2} \right\rangle & 2 & \left\langle 0 \right\rangle & 0 & \left\langle \varphi^{1\bar1\bar2} \right\rangle & 1 & \left\langle 0 \right\rangle & 0 & \left\langle \varphi^{1\bar1\bar2} \right\rangle & 1 \\
\midrule[0.02em]
(2,2) & \left\langle \varphi^{12\bar1\bar2} \right\rangle & 1 & \left\langle \varphi^{12\bar1\bar2} \right\rangle & 1 & \left\langle \varphi^{12\bar1\bar2} \right\rangle & 1 & \left\langle \varphi^{12\bar1\bar2} \right\rangle & 1 & \left\langle \varphi^{12\bar1\bar2} \right\rangle & 1 & \left\langle \varphi^{12\bar1\bar2} \right\rangle & 1 \\\bottomrule
\end{tabular}
}
\caption{Dolbeault and Bott-Chern cohomologies of compact complex surfaces diffeomorphic to solvmanifolds, part 2.}
\label{table:cohomologies-2}
\end{table}
\end{center}

\begin{center}
\begin{table}[ht]
 \centering
\resizebox{\textwidth}{!}{
\begin{tabular}{>{$\mathbf\bgroup}l<{\mathbf\egroup$} || >{$}l<{$} | >{$}c<{$} || >{$}l<{$} | >{$}c<{$} || >{$}l<{$} | >{$}c<{$} ||}
\toprule
 & \multicolumn{2}{c||}{\text{(A) torus}} & \multicolumn{2}{c||}{\text{(B) hyperelliptic}} & \multicolumn{2}{c||}{\text{(C) Inoue }$\mathcal{S}_M$} \\
k & H^{k}_{dR} & \dim_\C H^{k}_{dR} & H^{k}_{dR} & \dim_\C H^{k}_{dR} & H^{k}_{dR} & \dim_\C H^{k}_{dR} \\
\toprule
0 & \left\langle 1 \right\rangle & 1 & \left\langle 1 \right\rangle & 1 & \left\langle 1 \right\rangle & 1 \\
\midrule[0.02em]
1 & \left\langle \varphi^{1},\, \varphi^{2}, \, \varphi^{\bar1},\, \varphi^{\bar2} \right\rangle & 4 & \left\langle \varphi^{2},\, \varphi^{\bar2} \right\rangle & 2 & \left\langle \varphi^{2}-\varphi^{\bar2} \right\rangle & 1 \\
\midrule[0.02em]
2 & \left\langle \varphi^{12},\, \varphi^{1\bar1},\, \varphi^{1\bar2},\, \varphi^{2\bar1},\, \varphi^{2\bar2},\, \varphi^{\bar1\bar2} \right\rangle & 6 & \left\langle \varphi^{1\bar1},\, \varphi^{2\bar2} \right\rangle & 2 & \left\langle 0 \right\rangle & 0 \\
\midrule[0.02em]
3 & \left\langle \varphi^{12\bar1},\, \varphi^{12\bar2},\, \varphi^{1\bar1\bar2},\, \varphi^{2\bar1\bar2} \right\rangle & 4  & \left\langle \varphi^{12\bar1},\, \varphi^{1\bar1\bar2} \right\rangle & 2 & \left\langle \varphi^{12\bar1}-\varphi^{1\bar1\bar2} \right\rangle & 1 \\
\midrule[0.02em]
4 & \left\langle \varphi^{12\bar1\bar2} \right\rangle & 1 & \left\langle \varphi^{12\bar1\bar2} \right\rangle & 1 & \left\langle \varphi^{12\bar1\bar2} \right\rangle & 1 \\
\bottomrule
\end{tabular}
}
\caption{de Rham cohomology of compact complex surfaces diffeomorphic to solvmanifolds, part 1.}
\label{table:cohomologies-dR-1}
\end{table}
\end{center}

\begin{center}
\begin{table}[ht]
 \centering
\resizebox{\textwidth}{!}{
\begin{tabular}{>{$\mathbf\bgroup}l<{\mathbf\egroup$} || >{$}l<{$} | >{$}c<{$} || >{$}l<{$} | >{$}c<{$} || >{$}l<{$} | >{$}c<{$} ||}
\toprule
 & \multicolumn{2}{c||}{\text{(D) primary Kodaira}} & \multicolumn{2}{c||}{\text{(E) secondary Kodaira}} & \multicolumn{2}{c||}{\text{(F) Inoue $S^{\pm}$}} \\
k & H^{k}_{dR} & \dim_\C H^{k}_{dR} & H^{k}_{dR} & \dim_\C H^{k}_{dR} & H^{k}_{dR} & \dim_\C H^{k}_{dR} \\
\toprule
0 & \left\langle 1 \right\rangle & 1 & \left\langle 1 \right\rangle & 1 & \left\langle 1 \right\rangle & 1 \\
\midrule[0.02em]
1 & \left\langle \varphi^{1},\, \varphi^{\bar1},\, \varphi^{2}-\varphi^{\bar2} \right\rangle & 3 & \left\langle \varphi^{2}-\varphi^{\bar2} \right\rangle & 1 & \left\langle \varphi^{2}-\varphi^{\bar2} \right\rangle & 1 \\
\midrule[0.02em]
2 & \left\langle \varphi^{12},\, \varphi^{1\bar2},\, \varphi^{2\bar1},\, \varphi^{\bar1\bar2} \right\rangle & 4 & \left\langle 0 \right\rangle & 0 & \left\langle 0 \right\rangle & 0 \\
\midrule[0.02em]
3 & \left\langle \varphi^{12\bar2},\, \varphi^{2\bar1\bar2},\, \varphi^{12\bar1}-\varphi^{1\bar1\bar2} \right\rangle & 3 & \left\langle \varphi^{12\bar1}-\varphi^{1\bar1\bar2} \right\rangle & 1 & \left\langle \varphi^{12\bar1}-q\,\varphi^{12\bar2}-\varphi^{1\bar1\bar2}+q\,\varphi^{2\bar1\bar2} \right\rangle & 1 \\
\midrule[0.02em]
4 & \left\langle \varphi^{12\bar1\bar2} \right\rangle & 1 & \left\langle \varphi^{12\bar1\bar2} \right\rangle & 1 & \left\langle \varphi^{12\bar1\bar2} \right\rangle & 1 \\
\bottomrule
\end{tabular}
}
\caption{de Rham cohomology of compact complex surfaces diffeomorphic to solvmanifolds, part 2.}
\label{table:cohomologies-dR-2}
\end{table}
\end{center}

 In Tables \ref{table:cohomologies-1} and \ref{table:cohomologies-2} and in Tables \ref{table:cohomologies-dR-1} and \ref{table:cohomologies-dR-2}, we summarize the results of the computations. The sub-complexes of left-invariant forms are depicted in Figure \ref{fig:left-inv} (each dot represents a generator, vertical arrows depict the $\delbar$-operator, horizontal arrows depict the $\del$-operator, and trivial arrows are not shown.) The dimensions are listed in Table \ref{table:cohom}.

 On the one side, recall that the inclusion of left-invariant forms into the space of forms induces an injective map in Dolbeault and Bott-Chern cohomologies, see, e.g., \cite[Lemma 9]{console-fino}, \cite[Lemma 3.6]{angella-1}. On the other side, recall that the Fr\"olicher spectral sequence of a compact complex surface $X$ degenerates at the first level, equivalently, the equalities
 $$ \dim_\C H^{1,0}_{\delbar}(X) + \dim_\C H^{0,1}_{\delbar}(X) \;=\; \dim_\C H^1_{dR}(X;\C) $$
 and
 $$ \dim_\C H^{2,0}_{\delbar}(X) + \dim_\C H^{1,1}_{\delbar}(X) + \dim_\C H^{0,2}_{\delbar}(X) \;=\; \dim_\C H^2_{dR}(X;\C) $$
 hold. By comparing the dimensions in Table \ref{table:cohom} with the Betti numbers case by case, we find that the left-invariant forms suffice in computing the Dolbeault cohomology for each case. Then, by \cite[Theorem 3.7]{angella-1}, see also \cite[Theorem 1.3, Theorem 1.6]{angella-kasuya-1}, it follows that also the Bott-Chern cohomology is computed using just left-invariant forms.
\end{proof}

\begin{figure}[hpt!]
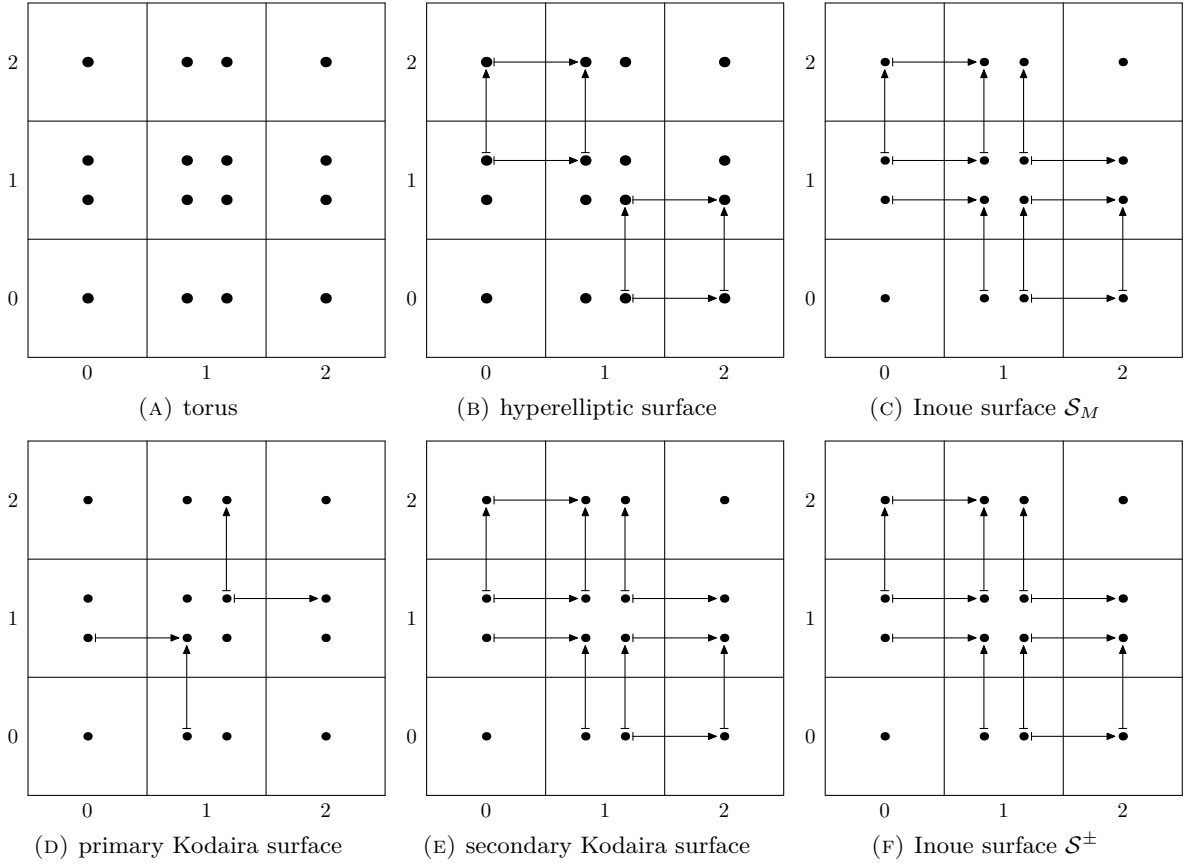

        \centering
        \begin{subfigure}[b]{0.3\textwidth}
                \centering
                 \includegraphics[width=5cm]{./coom.1}
                \caption{torus}
                \label{fig:torus}
        \end{subfigure}
        \quad
        \begin{subfigure}[b]{0.3\textwidth}
                \centering
                \includegraphics[width=5cm]{./coom.2}
                \caption{hyperelliptic surface}
                \label{fig:hyperell}
        \end{subfigure}
	\quad
        \begin{subfigure}[b]{0.3\textwidth}
                \centering
                \includegraphics[width=5cm]{./coom.3}
                \caption{Inoue surface $\mathcal{S}_M$}
                \label{fig:inoue-0}
        \end{subfigure}

	\medskip

        \begin{subfigure}[b]{0.3\textwidth}
                \centering
                 \includegraphics[width=5cm]{./coom.4}
                \caption{primary Kodaira surface}
                \label{fig:prim-kodaira}
        \end{subfigure}
        \quad
        \begin{subfigure}[b]{0.3\textwidth}
                \centering
                \includegraphics[width=5cm]{./coom.5}
                \caption{secondary Kodaira surface}
                \label{fig:sec-kodaira}
        \end{subfigure}
	\quad
        \begin{subfigure}[b]{0.3\textwidth}
                \centering
                \includegraphics[width=5cm]{./coom.6}
                \caption{Inoue surface $\mathcal{S}^\pm$}
                \label{fig:inoue-pm}
        \end{subfigure}
        \caption{The double-complexes of left-invariant forms over $4$-dimensional solvmanifolds.}\label{fig:left-inv}
\end{figure}

\begin{center}
\begin{table}[ht]
 \centering
\resizebox{\textwidth}{!}{
\begin{tabular}{>{$\mathbf\bgroup}l<{\mathbf\egroup$} || >{$}p{.5cm}<{$} >{$}p{.5cm}<{$} >{$}p{.5cm}<{$} >{$}p{.5cm}<{$} | >{$}p{.5cm}<{$} >{$}p{.5cm}<{$} >{$}p{.5cm}<{$} >{$}p{.5cm}<{$} | >{$}p{.5cm}<{$} >{$}p{.5cm}<{$} >{$}p{.5cm}<{$} >{$}p{.5cm}<{$} | >{$}p{.5cm}<{$} >{$}p{.5cm}<{$} >{$}p{.5cm}<{$} >{$}p{.5cm}<{$} | >{$}p{.5cm}<{$} >{$}p{.5cm}<{$} >{$}p{.5cm}<{$} >{$}p{.5cm}<{$} | >{$}p{.5cm}<{$} >{$}p{.5cm}<{$} >{$}p{.5cm}<{$} >{$}p{.5cm}<{$} ||}
\toprule
 & \multicolumn{4}{c|}{(A) torus} & \multicolumn{4}{c|}{(B) hyperell} & \multicolumn{4}{c|}{(C) Inoue $\mathcal{S}_M$} & \multicolumn{4}{c|}{(D) prim Kod} & \multicolumn{4}{c|}{(E) sec Kod} & \multicolumn{4}{c||}{(F) Inoue $\mathcal{S}^\pm$} \\
(p,q) & h^{p,q}_{\delbar} & h^{p,q}_{BC} & b_k & \Delta^k & h^{p,q}_{\delbar} & h^{p,q}_{BC} & b_k & \Delta^k & h^{p,q}_{\delbar} & h^{p,q}_{BC} & b_k & \Delta^k & h^{p,q}_{\delbar} & h^{p,q}_{BC} & b_k & \Delta^k & h^{p,q}_{\delbar} & h^{p,q}_{BC} & b_k & \Delta^k & h^{p,q}_{\delbar} & h^{p,q}_{BC} & b_k & \Delta^k \\
\toprule
(0,0) & 1 & 1 & 1 & 0 & 1 & 1 & 1 & 0 & 1 & 1 & 1 & 0 & 1 & 1 & 1 & 0 & 1 & 1 & 1 & 0 & 1 & 1 & 1 & 0 \\
\midrule[0.02em]
(1,0) & 2 & 2 & \multirow{2}{*}{4} & \multirow{2}{*}{0} & 1 & 1 & \multirow{2}{*}{2} & \multirow{2}{*}{0} & 0 & 0 & \multirow{2}{*}{1} & \multirow{2}{*}{0} & 1 & 1 & \multirow{2}{*}{3}& \multirow{2}{*}{0} & 0 & 0 & \multirow{2}{*}{1} & \multirow{2}{*}{0} & 0 & 0 & \multirow{2}{*}{1} & \multirow{2}{*}{0} \\
(0,1) & 2 & 2 & & & 1 & 1 & & & 1 & 0 & & & 2 & 1 & & & 1 & 0 & & & 1 & 0 & & \\
\midrule[0.02em]
(2,0) & 1 & 1 & \multirow{2}{*}{6} & \multirow{2}{*}{0} & 0 & 0 & \multirow{2}{*}{2} & \multirow{2}{*}{0} & 0 & 0 & \multirow{2}{*}{0} & \multirow{2}{*}{2} & 1 & 1 & \multirow{2}{*}{4} & \multirow{2}{*}{2} & 0 & 0 & \multirow{2}{*}{0} & \multirow{2}{*}{2} & 0 & 0 & \multirow{2}{*}{0} & \multirow{2}{*}{2} \\
(1,1) & 4 & 4 & & & 2 & 2 & & & 0 & 1 & & & 2 & 3 & & & 0 & 1 & & & 0 & 1 & & \\
(0,2) & 1 & 1 & & & 0 & 0 & & & 0 & 0 & & & 1 & 1 & & & 0 & 0 & & & 0 & 0 & & \\
\midrule[0.02em]
(2,1) & 2 & 2 & \multirow{2}{*}{4} & \multirow{2}{*}{0} & 1 & 1 & \multirow{2}{*}{2} & \multirow{2}{*}{0} & 1 & 1 & \multirow{2}{*}{1} & \multirow{2}{*}{0} & 2 & 2 & \multirow{2}{*}{3} & \multirow{2}{*}{0} & 1 & 1 & \multirow{2}{*}{1} & \multirow{2}{*}{0} & 1 & 1 & \multirow{2}{*}{1} & \multirow{2}{*}{0} \\
(1,2) & 2 & 2 & & & 1 & 1 & & & 0 & 1 & & & 1 & 2 & & & 0 & 1 & & & 0 & 1 & & \\
\midrule[0.02em]
(2,2) & 1 & 1 & 1 & 0 & 1 & 1 & 1 & 0 & 1 & 1 & 1 & 0 & 1 & 1 & 1 & 0 & 1 & 1 & 1 & 0 & 1 & 1 & 1 & 0 \\
\bottomrule
\end{tabular}
}
\caption{Summary of the dimensions of de Rham, Dolbeault, and Bott-Chern cohomologies and of the degree of non-K\"ahlerness for compact complex surfaces diffeomorphic to solvmanifolds.}
\label{table:cohom}
\end{table}
\end{center}

\medskip

We prove the following result.

\begin{thm}\label{thm:inj-solvmfds}
Let $X$ be a compact complex surface diffeomorphic to a solvmanifold. Then the natural map $H^{2,1}_{BC}(X)\to H^{2,1}_{\delbar}(X)$ induced by the identity is an isomorphism, and the natural map $H^{2,1}_{BC}(X)\to H^{3}_{dR}(X;\C)$ induced by the identity is injective.
\end{thm}

\begin{proof}
By the general result in Theorem \ref{thm:delta1=0}, the natural map $H^{2,1}_{BC}(X)\to H^{2,1}_{\delbar}(X)$ is injective. In fact, it is an isomorphism as follows from the computations summarized in Tables \ref{table:cohomologies-1} and \ref{table:cohomologies-2}.
As for the injectivity of the natural map $H^{2,1}_{BC}(X)\to H^{3}_{dR}(X;\C)$, it is a straightforward computation from Tables \ref{table:cohomologies-1} and \ref{table:cohomologies-2} and Tables \ref{table:cohomologies-dR-1} and \ref{table:cohomologies-dR-2}.

As an example, we offer an explicit calculation of the injectivity of the map $H^{2,1}_{BC}(X)\to H^{3}_{dR}(X;\C)$ for the Inoue surfaces of type $0$, see \cite{inoue}, see also \cite{tricerri}. We will change a little bit the notation. Recall the construction of Inoue surfaces: let $M\in \mathrm{SL}(3;\Z)$ be a unimodular matrix having a real eigenvalue $\lambda >1$ and two complex 
eigenvalues $\mu\neq\overline{\mu}$. Take a real eigenvector $(\alpha_1,\alpha_2,\alpha_3)$ and an eigenvector 
$(\beta_1,\beta_2,\beta_3)$ of $M$. Let $\mathbb{H}=\{z\in\C\,\,\,\vert\,\,\, \Im\mathfrak{m}\,z>0\}$; on the product $\mathbb{H}\times \C$ consider
the following transformations defined as 
\begin{eqnarray*}
f_0(z,w) &:=& (\lambda z,\, \mu w)\\
f_j(z,w) &:=& (z+\alpha_j,w+\beta_j) \quad \text{ for } \quad j \in \{1,2,3\} \;.
\end{eqnarray*} 
Denote by $\Gamma_M$ the group generated by $f_0,\ldots,f_3$; then $\Gamma_M$ acts in a properly discontinuous way and without fixed points on $\mathbb{H}\times \C$, and $\mathcal{S}_{M} := \mathbb{H}\times \C\slash \Gamma_M$ is an Inoue surface of type $0$, as in case (C) in \cite{hasegawa-jsg}. Denoting by $z=x+\im y$ and $w=u+\im v$, consider the following differential forms on $\mathbb{H}\times\mathbb{C}$:
$$ e^1 \;:=\; \frac{1}{y}\, \de x \;, \quad e^2 \;:=\; \frac{1}{y}\, \de y \;, \quad e^3 \;:=\; \sqrt{y}\, \de u \;, \quad e^4 \;:=\; \sqrt{y}\, \de v \;. $$
(Note that $e^1$ and $e^2$, and $e^3\wedge e^4$ are $\Gamma_M$-invariant, and consequently they induce global differential forms on $\mathcal{S}_{M}$.) We obtain
$$ \de e^1 \;=\; e^1\wedge e^2 \;, \quad \de e^2 \;=\; 0 \;, \quad \de e^3 \;=\; \frac{1}{2}\, e^2 \wedge e^3 \;, \quad \de e^4 \;=\; \frac{1}{2}\, e^2\wedge e^4 \;. $$
Consider the natural complex structure on $\mathcal{S}_{M}$ induced by $\mathbb{H}\times\C$. Locally, we have
$$ J e^1 = - e^2 \quad \text{ and } \quad J e^2 = e^1 \quad \text{ and } \quad  J e^3 = -e^4 \quad \text{ and } \quad J e^4 = e^3 \;. $$
Considering the $\Gamma_M$-invariant $(2,1)$-Bott-Chern cohomology of $\mathcal{S}_{M}$, we obtain that
$$ H^{2,1}_{BC}(\mathcal{S}_{M}) \;=\; \C\left\langle \left[ e^1\wedge e^3\wedge e^4 + \im e^2\wedge e^3\wedge e^4 \right] \right\rangle \;. $$
Clearly $\delbar \left( e^1\wedge e^3\wedge e^4 + \im e^2\wedge e^3\wedge e^4 \right) = 0$ and $e^1\wedge e^3\wedge e^4 + \im e^2\wedge e^3\wedge e^4 = e^{1} \wedge e^3 \wedge e^4 + \im\, \de \left( e^3 \wedge e^4 \right)$, therefore the de Rham cohomology class $\left[ e^1\wedge e^3\wedge e^4 + \im e^2\wedge e^3\wedge e^4 \right] = \left[ e^1 \wedge e^3 \wedge e^4 \right] \in H^3_{dR}(\mathcal{S}_{M})$ is non-zero.
\end{proof}


\begin{thebibliography}{10}

\bibitem{angella-1}
D. Angella, The cohomologies of the Iwasawa manifold and of its small deformations, {\em J. Geom. Anal.} \textbf{23} (2013), no.~3, 1355-1378.

\bibitem{angella-kasuya-1}
D. Angella, H. Kasuya, Bott-Chern cohomology of solvmanifolds, \texttt{arXiv:1212.5708v3 [math.DG]}.

\bibitem{angella-tomassini-3}
D. Angella, A. Tomassini, On the $\partial\overline\partial$-Lemma and Bott-Chern cohomology, {\em Invent. Math.} \textbf{192} (2013), no.~1, 71--81.

\bibitem{barth-hulek-peters-vandeven}
W.~P. Barth, K. Hulek, C.~A.~M. Peters, A. Van de Ven, {\em Compact complex surfaces}, Second edition, Ergebnisse der Mathematik und ihrer Grenzgebiete, \textbf{3}, Springer-Verlag, Berlin, 2004.

\bibitem{buchdahl}
N. Buchdahl, On compact Kähler surfaces, {\em Ann. Inst. Fourier (Grenoble)} \textbf{49} (1999), no.~1, vii, xi, 287--302.

\bibitem{calabi-eckmann}
E. Calabi, B. Eckmann, A class of compact, complex manifolds which are not algebraic, {\em Ann. of Math. (2)} \textbf{58} (1953), no.~3, 494--500.

\bibitem{console-fino}
S. Console, A. Fino, Dolbeault cohomology of compact nilmanifolds, {\em Transform. Groups} \textbf{6} (2001), no.~2, 111--124.

\bibitem{deligne-griffiths-morgan-sullivan}
P. Deligne, Ph. Griffiths, J. Morgan, D.~P. Sullivan, Real homotopy theory of K\"ahler manifolds, {\em Invent. Math.} \textbf{29} (1975), no.~3, 245--274.

\bibitem{dloussky-ajm}
G. Dloussky, On surfaces of class $VII^+_0$ with numerically anticanonical divisor, {\em Amer. J. Math.} \textbf{128} (2006), no.~3, 639--670.

\bibitem{fu-yau}
J. Fu, S.-T. Yau, A note on small deformations of balanced manifolds, {\em C. R. Math. Acad. Sci. Paris} \textbf{349} (2011), no.~13--14, 793--796.

\bibitem{hasegawa-jsg}
K. Hasegawa, Complex and K\"ahler structures on compact solvmanifolds, Conference on Symplectic Topology, {\em J. Symplectic Geom.} \textbf{3} (2005), no.~4, 749--767.

\bibitem{hirzebruch}
F. Hirzebruch, {\em Topological methods in algebraic geometry}, Translated from the German and Appendix One by R. L. E. Schwarzenberger, With a preface to the third English edition by the author and Schwarzenberger, Appendix Two by A. Borel, Reprint of the 1978 edition, Classics in Mathematics, Springer-Verlag, Berlin, 1995.

\bibitem{inoue}
M. Inoue, On surfaces of Class $VII_0$, {\em Invent. Math.} \textbf{24} (1974), no.~4, 269--310.

\bibitem{kodaira-1}
K. Kodaira, On the structure of compact complex analytic surfaces. I, {\em Amer. J. Math.} \textbf{86} (1964), 751--798.

\bibitem{lamari}
A. Lamari, Courants kählériens et surfaces compactes, {\em Ann. Inst. Fourier (Grenoble)} \textbf{49} (1999), no.~1, vii, x, 263--285.

\bibitem{lubke-teleman}
M. L\"ubke, A. Teleman, {\em The Kobayashi-Hitchin correspondence}, World Scientific Publishing Co., Inc., River Edge, NJ, 1995.

\bibitem{miyaoka}
Y. Miyaoka, Kähler metrics on elliptic surfaces, {\em Proc. Japan Acad.} \textbf{50} (1974), no.~8, 533--536.

\bibitem{parton-rend}
M. Parton, Explicit parallelizations on products of spheres and Calabi-Eckmann structures, {\em Rend. Istit. Mat. Univ. Trieste} \textbf{35} (2003), no.~1-2, 61--67.

\bibitem{schweitzer}
M. Schweitzer, Autour de la cohomologie de Bott-Chern, \texttt{arXiv:0709.3528v1 [math.AG]}.

\bibitem{siu}
Y.~T. Siu, Every K3 surface is Kähler, {\em Invent. Math.} \textbf{73} (1983), no.~1, 139--150.

\bibitem{teleman}
A. Teleman, The pseudo-effective cone of a non-Kählerian surface and applications, {\em Math. Ann.} \textbf{335} (2006), no.~4, 965--989.

\bibitem{tricerri}
F. Tricerri, Some examples of locally conformal Kähler manifolds, {\em Rend. Sem. Mat. Univ. Politec. Torino} \textbf{40} (1982), no.~1, 81--92.

\end{thebibliography}
\end{document}